\documentclass[11pt,a4paper]{article}
\pdfoutput=1
\usepackage[svgnames,hyperref]{xcolor}
\colorlet{siaminlinkcolor}{green!50!black}
\colorlet{siamexlinkcolor}{red!50!black}

\usepackage[colorlinks,
	allcolors = siaminlinkcolor,
	urlcolor = siamexlinkcolor,
]{hyperref}
\usepackage[utf8]{inputenc}
\usepackage{amsmath,amsthm,amssymb}
\usepackage[T1]{fontenc}
\usepackage{lmodern}
\usepackage{microtype}
\usepackage[nameinlink,capitalise]{cleveref}
\usepackage[DIV=11]{typearea}

\usepackage{pgfplots}
\pgfplotsset{compat=1.16}

\numberwithin{equation}{section}
\theoremstyle{plain}
\newtheorem{theorem}{Theorem}[section]
\newtheorem{proposition}[theorem]{Proposition}
\newtheorem{lemma}[theorem]{Lemma}
\newtheorem{corollary}[theorem]{Corollary}

\theoremstyle{definition}

\newtheorem{assumption}[theorem]{Assumption}
\newtheorem{remark}[theorem]{Remark}

\Crefname{assumption}{Assumption}{Assumptions}
\Crefname{remark}{Remark}{Remarks}
\Crefname{example}{Example}{Examples}
\Crefname{figure}{Figure}{Figures}

\newcommand{\N}{\mathbb{N}}
\newcommand{\R}{\mathbb{R}}
\newcommand{\Rd}{{\R^d}}
\newcommand{\Normal}{\mathcal{N}}
\newcommand{\Laplace}{\mathcal{L}}

\newcommand{\norm}[1]{{\lVert #1 \rVert}}
\newcommand{\eunorm}[1]{{\lvert #1 \rvert}}
\newcommand{\bigeunorm}[1]{\left\lvert #1 \right\rvert}
\newcommand{\bignorm}[1]{\left\lVert #1 \right\rVert}
\newcommand{\abs}[1]{{\lvert #1 \rvert}}
\newcommand{\bigabs}[1]{\left\lvert #1 \right\rvert}
\newcommand{\scalprod}[1]{{(#1)}}
\newcommand{\bigscalprod}[1]{\left(#1\right)}
\newcommand{\Exp}{\mathbb{E}}
\newcommand{\Prob}{\mathbb{P}}

\newcommand{\dTV}{{d_\text{TV}}}
\newcommand{\dH}{d_\text{H}}
\newcommand{\dBL}{d_\text{BL}}
\newcommand{\Borel}{\mathcal{B}}

\newcommand{\di}{\mathrm{d}}
\newcommand{\Lmuy}{{\Laplace_{\mu^y}}}
\newcommand{\xmap}{\hat{x}}

\newcommand{\T}{{\mathrm{T}}}

\newcommand\derbound{K}

\renewcommand{\epsilon}{\varepsilon}

\DeclareMathOperator{\tr}{tr}

\newcommand{\email}[1]{\href{mailto:#1}{#1}}

\title{\sc non-asymptotic error estimates for the laplace approximation in bayesian inverse problems}
\author{%
	Tapio Helin\thanks{LUT University, School of Engineering Science, P.O.~Box 20, FI-53851 Lappeenranta, Finland (\email{tapio.helin@lut.fi}, \email{remo.kretschmann@mathematik.uni-wuerzburg.de}). The work of the authors was supported by the Academy of Finland (decision 326961). RK is now affiliated with the Institute of Mathematics, University of Würzburg, Emil-Fischer-Str. 30, 97074 Würzburg, Germany.}
	\and 
	Remo Kretschmann\footnotemark[1]
}
\date{27 October 2021}
\hypersetup{
    pdftitle = {Non-asymptotic error estimates for the Laplace approximation in Bayesian inverse problems},
    pdfauthor = {Tapio Helin, Remo Kretschmann},
    pdfkeywords = {Laplace approximation, Gaussian approximation, total variation distance, Bayesian inference, nonlinear ill-posed problems},
}

\begin{document}

\maketitle

\begin{abstract}
In this paper we study properties of the Laplace approximation of the posterior distribution arising in nonlinear Bayesian inverse problems.
Our work is motivated by Schillings et al.~(2020), where it is shown that in such a setting the Laplace approximation error in Hellinger distance converges to zero in the order of the noise level.
Here, we prove novel error estimates for a given noise level that also quantify the effect due to the nonlinearity of the forward mapping and the dimension of the problem. In particular, we are interested in settings in which a linear forward mapping is perturbed by a small nonlinear mapping. Our results indicate that in this case, the Laplace approximation error is of the size of the perturbation. The paper provides insight into Bayesian inference in nonlinear inverse problems, where linearization of the forward mapping has suitable approximation properties.
\end{abstract}

\section{Introduction}

The study of Bayesian inverse problems \cite{kaipio2006statistical, stuart2010inverse} has gained wide attention during the last decade as the increase in computational resources and algorithmic development have enabled uncertainty quantification in numerous new applications in science and engineering. Large-scale problems, where the computational burden of the likelihood is prohibitive, are, however, still a subject of ongoing research.

In this paper we study the Laplace approximation of the posterior distribution arising in nonlinear Bayesian inverse problems. The Laplace approximation is obtained by replacing the log-posterior density with its second order Taylor approximation around the maximum a posteriori (MAP) estimate and renormalizing the density. This produces a Gaussian measure centered at the maximum a posteriori (MAP) estimate with a covariance corresponding to the Hessian of the negative log-posterior density (see, e.g., \cite[Section 4.4]{bishop2006pattern}).

The asymptotic behavior of the parametric Laplace approximation in the small noise or large data limit has been studied extensively in the past (see, e.g., \cite{wong2001asymptotic}).
We note that in terms of approximation properties with respect to taking a posterior expectation over a given function, there is a long line of research which we discuss below. Our work is parallel to this effort in that we aim to estimate the total variation (TV) distance between the two probability measures. 
On the one hand, the error in TV distance bounds the error of the expectation of any function with respect to the Laplace approximation.
On the other hand, it is a measure of the non-Gaussianity of the posterior distribution. Thus, our results describe and quantify how the nonlinearity of the forward mapping translates into non-Gaussianity of the posterior distribution.

Our work is motivated by a recent result by Schillings, Sprungk, and Wacker in \cite{SSW:2020}, where the authors show that in the context of Bayesian inverse problems, the Laplace approximation error in Hellinger distance converges to zero in the order of the noise level.
In practice, one is, however, often interested in estimating the error for a given, \emph{fixed} noise level. It can, e.g., be unclear if the noise level is small enough in order to dominate the error estimate.
Indeed, the nonlinearity of the forward mapping (more generally, the non-Gaussianity of the likelihood) or a large problem dimension can have a signifact contribution to the constant appearing in the asymptotic estimates.
Therefore, it is of interest to quantify such effects in non-asymptotic error estimates for the Laplace approximation. This is the main goal of our work.

\subsection{Our contributions}

The main contribution of this work is threefold:

\begin{enumerate}
	\item In \cref{main_estimate}, we derive our central error estimate for the total variation distance of the Laplace posterior approximation in nonlinear Bayesian inverse problems. The error bound consists of two error terms for which we derive an implicit optimal balancing rule in \cref{prop:optimality}.
We assume uniform bounds on the third differentials of log-likelihood and log-prior density as well as a quadratic lower bound on the log-posterior density to control the error.
Given such bounds the error estimate can be numerically evaluated.
	\item In \cref{nonasymp_est_dim}, we derive a further estimate for the Laplace approximation error that makes the effect of noise level, the bounds specified above and the dimension of the problem explicit. This error estimate readily implies linear rates of convergence for fixed problem dimension both in the small noise limit and when the third differential of the log-likelihood goes to zero, see \cref{conv_rate_fixed_dim}. It furthermore leads to a convergence rate for increasing problem dimension in terms of noise level, problem dimension, and aforementioned bounds aligned with \cite{Lu:2017}, see \cref{cor_inc_prob_dim}.
	\item In \cref{conv_rate_nonlin_perturb}, we quantify the error of the Laplace approximation in terms of the nonlinearity of the forward mapping for linear inverse problems with small nonlinear perturbation and Gaussian prior distribution. We assume uniform bounds on the differentials of the nonlinear perturbation of up to third order to control the error. This error estimate immediately implies linear convergence in terms of the size of the perturbation. Moreover, such a result provides insight into Bayesian inference in nonlinear inverse problems, where linearization of the forward mapping has suitable approximation properties.
\end{enumerate}

\subsection{Relevant literature}

The asymptotic approximation of general integrals of the form $\int e^{\lambda f(x)}g(x) \di x$ by Laplace's method is presented in \cite{olver1974asymptotics,wong2001asymptotic}.
Non-asymptotic error bounds for the Laplace approximation of such integrals have been stated in the univariate \cite{Olv:1968} and multivariate case \cite{MW:1983,IM:2014}. 
The Laplace approximation error and its convergence in the limit $\lambda \to \infty$ have been estimated in the multivariate case when the function $f$ depends on $\lambda$ or the maximizer of $f$ is on the boundary of the integration domain \cite{lapinski2019multivariate}.
A representation of the coefficients appearing in the asymptotic expansion of the approximated integral utilizing ordinary potential polynomials is given in \cite{nemes2013explicit}.

The error estimates on the Laplace approximation in TV distance are closely connected to the so-called Bernstein--von Mises (BvM) phenomenon that quantifies the convergence of the scaled posterior distribution toward a Gaussian distribution in the large data or small noise limit.
Parametric BvM theory is well-understood \cite{van2000asymptotic, le2012asymptotic}. Our work is inspired by a BvM result by Lu in \cite{Lu:2017}, where a parametric BvM theorem for nonlinear Bayesian inverse problems with an increasing number of parameters is proved. Similar to our objectives, he quantifies the asymptotic convergence rate in terms of noise level, nonlinearity of the forward mapping and dimension of the problem. 
However, our emphasis differs from \cite{Lu:2017} (and other BvM results) in that we are not restricted to considering the vanishing noise limit, but are more interested in quantifying the effect of small nonlinearity or dimension at a fixed noise level.
We also point out that BvM theory has been developed for non-parametric Bayesian inverse problems (see, e.g., \cite{nickl2020, monard2019efficient, giordano2020bernstein}), where the convergence is quantified in a distance that metrizes the weak convergence.

Let us conclude by briefly emphasizing that the Laplace approximation is widely utilized for different purposes in computational Bayesian statistics including, i.a., the celebrated INLA algorithm \cite{rue2009approximate}.
It has also recently gained popularity in optimal Bayesian experimental design (see, e.g., \cite{ryan2016review, long2013fast, alexanderian2016fast}). Moreover, it provides a convenient reference measure for numerical quadrature \cite{chen2017hessian, schillings2016scaling} or importance sampling \cite{beck2018fast}.

\subsection{Organization of the paper}

Before we present the aforementioned three main results in \cref{sec:estimate,sec:param_free_est,sec:small_nonlin_limit}, we introduce our set-up and notation, Laplace's method, and the total variation metric in \cref{sec:prelim_setup}.
In \cref{sec:estimate}, we introduce our central error bound for the Laplace approximation and explain the idea behind its proof.
In \cref{sec:param_free_est}, we derive an explicit error estimate for the Laplace approximation and describe its asymptotic behavior.
In \cref{sec:small_nonlin_limit}, we prove the error estimate for inverse problems with small nonlinearity in the forward mapping and Gaussian prior distribution.


\section{Preliminaries and set-up}
\label{sec:prelim_setup}

We consider for $\varepsilon > 0$ the inverse problem of recovering $x \in \R^d$ from a noisy measurement $y \in \R^d$, where
\[ y = G(x) + \sqrt{\varepsilon}\eta, \]
$\eta \in \R^d$ is random noise with standard normal distribution $\Normal(0,I_d)$, and $G$: $\Rd \to \Rd$ is a possibly nonlinear mapping. 
In the following, $\eunorm{\cdot}$ denotes the Euclidean norm on $\R^d$.
If we assume a prior distribution $\mu$ on $\R^d$ with Lebesgue density $\exp(-R(x))$, then Bayes' formula yields a posterior distribution $\mu^y$ with density
\begin{equation}
	\label{post_dens_eps}
		\mu^y(\di x) \propto \exp \left( -\frac{1}{2\varepsilon} \eunorm{y - G(x)}^2 \right) \mu(\di x) 
		= \exp \left( -\frac{1}{2\varepsilon} \eunorm{y - G(x)}^2 - R(x) \right) \di x.
\end{equation}
For all $x, y \in \R^d$, we denote the \emph{scaled} negative log-likelihood by
\begin{equation*}
	\Phi(x) = \frac12\eunorm{y - G(x)}^2.
\end{equation*}
If 
\[
	x \mapsto \Phi(x) + \varepsilon R(x)
\]
has a unique minimizer in $\R^d$, we call this minimizer the \emph{maximum a posteriori (MAP) estimate} and denote it by $\xmap = \xmap(y)$.
Furthermore, we set
\begin{equation*}
	I(x) := \Phi(x) + \varepsilon R(x) - \Phi(\xmap) - \varepsilon R(\xmap)
\end{equation*}
for all $x \in \R^d$.
This way, $I$ is nonnegative, the MAP estimate $\xmap$ minimizes $I$ and satisfies $I(\xmap) = 0$.
Moreover, we can express the posterior density as
\begin{equation}
	\label{post_dens_I}
	\mu^y(\di x) = \frac{1}{Z} \exp \left( -\frac{1}{\varepsilon}I(x) \right) \di x
\end{equation}
with a normalization constant $Z$.

Laplace's method approximates the posterior distribution by a Gaussian distribution $\Lmuy$ whose mean and covariance are chosen in such a way that its log-density agrees, up to a constant, with the second order Taylor polynomial around $\xmap$ of the log-posterior density.
If $I \in C^2(\Rd,\R)$, the Laplace approximation of $\mu^y$ is defined as 
\begin{equation*}
	\Laplace_{\mu^y} := \Normal(\hat{x}, \varepsilon \Sigma),
\end{equation*}
where $\Sigma := (D^2I(\hat{x}))^{-1}$. Here, $DI$ denotes the differential of $I$, and we identify $D^2I(\xmap)$ with the Hessian matrix $\{D^2I(\xmap)(e_j,e_k)\}_{j,k = 1}^d$.
The Lebesgue density of $\Lmuy$ is given by
\begin{align*}
	\Laplace_{\mu^y}(\di x) 
	&= \frac{1}{\widetilde{Z}} \exp\left(-\frac{1}{2\varepsilon}\norm{x - \hat{x}}_\Sigma^2\right) \di x \\
	&= \frac{1}{\widetilde{Z}} \exp\left(-\frac{1}{2\varepsilon}D^2 I(\hat{x})(x - \hat{x},x - \hat{x})\right) \di x,
\end{align*}
where
\begin{equation}
	\label{form_Z_tilde}
	\widetilde{Z} = \int_{\R^d} \exp\left(-\frac{1}{2\varepsilon}\norm{x - \hat{x}}_\Sigma^2\right) \di x 
	= \varepsilon^{\frac{d}{2}} (2\pi)^\frac{d}{2} \sqrt{\det \Sigma}.
\end{equation}
Since $I(\xmap) = 0$ and $DI(\xmap) = 0$, $\frac{1}{2\epsilon}\norm{x - \xmap}_\Sigma^2$ is precisely the truncated Taylor series of $I/\varepsilon$ around $\hat{x}$.

The \emph{total variation (TV) distance} between two probability measures $\nu$ and $\mu$ on $(\Rd,\Borel(\Rd))$ is defined as
\begin{equation*}
	\label{def_TV_dist}
	\dTV(\nu,\mu) = \sup_{A \in \Borel(\R^d)} \abs{\nu(A) - \mu(A)},
\end{equation*}
see Section 2.4 in \cite{sullivan2015}.
It has the alternative representation
\begin{equation*}
	\label{equi_def_TV_dist}
	\dTV(\nu,\mu) = \frac12 \sup_{\norm{f}_\infty \le 1} \left\{\int_\Rd f \di\nu - \int_\Rd f \di\mu\right\} = \frac12 \int_{\R^d} \left| \frac{d\nu}{d\rho} - \frac{d\mu}{d\rho} \right| d\rho
\end{equation*}
where $\norm{f}_\infty := \sup_{x \in \Rd} \abs{f(x)}$ and $\rho$ can be any probability measure dominating both $\mu$ and $\nu$, see Remark 5.9 in \cite{sullivan2015} and equation (1.12) in \cite{law2015}.
The total variation distance is valuable for the purpose of uncertainty quantification because it bounds the error of any credible region when using a measure $\nu$ instead of another measure $\mu$.
It can, moreover, be used to bound the difference in expectation of any bounded function $f$ on $\Rd$ with respect to $\mu$ and $\nu$, respectively, by
\[
	\bigabs{\Exp^\nu[f] - \Exp^\mu[f]} \le 2\norm{f}_\infty\dTV(\nu,\mu),
\]
see Lemma 1.32 in \cite{law2015}.
By Kraft's inequality
\[
	\dH(\mu,\nu)^2 \le \dTV(\mu,\nu) \le \sqrt{2}\dH(\mu,\nu),
\]
the total variation distance bounds the square of the \emph{Hellinger distance}
\[
	\dH(\mu,\nu) = \left( \frac12 \int_{\R^d} \left| \sqrt{\frac{d\nu}{d\rho}} - \sqrt{\frac{d\mu}{d\rho}} \right|^2 d\rho \right)^\frac12,
\]
see Definition 1.28 and Lemma 1.29 in \cite{law2015} or \cite{kraft1955}, while both metrics induce the same topology.
The \emph{bounded Lipschitz metric}
\[
	\dBL(\mu,\nu) = \frac12\sup_{\norm{f}_\infty + \norm{f}_\text{Lip} \le 1} \left\{\int_\Rd f \di\mu - \int_\Rd f \di\nu\right\},
\]
which induces the topology of weak convergence of probability measures, is trivially bounded by the total variation distance. Here, we denote
\[
	\norm{f}_\text{Lip} := \sup_{x,y \in \Rd,\,x \neq y} \frac{\abs{f(x) - f(y)}}{\eunorm{x - y}}.
\]
For further information on the relation between the total variation distance and other probability metrics we refer the survey paper \cite{gibbs2002}.


\section{Central error estimate}
\label{sec:estimate}

We will use the following ideas to bound the error of the Laplace approximation $\Laplace_{\mu^y}$ for a given realization of the data $y \in \R^d$.
First, we will prove the fundamental estimate
\begin{equation}
	\label{dTV_fund_int}
	\dTV(\mu^y,\Lmuy) 
	\le \frac{1}{\widetilde{Z}} \int_{\R^d} \bigabs{\exp\left(-\frac{1}{\epsilon}I(x)\right) - \exp\left(-\frac{1}{2\epsilon}\norm{x - \xmap}_\Sigma^2\right)} \di x.
\end{equation}
If we have a radial upper bound $f(\norm{x - \xmap}_\Sigma)$ for the integrand on the right hand side of \eqref{dTV_fund_int}, we can estimate
\[
	\dTV(\mu^y,\Lmuy) \le \int_\Rd f(\norm{x - \xmap}_\Sigma) \di x = \sqrt{\det \Sigma} \int_\Rd f(\eunorm{u}) \di u,
\]
where we applied a change of variable to a local parameter $u := \Sigma^{-\frac12}(x - \xmap)$.
This integral, we can now express as a $1$-dimensional integral using polar coordinates.

The integrand on the right hand side of \eqref{dTV_fund_int} is very small and flat around $\xmap$, since $\frac12\norm{x - \xmap}_\Sigma^2$ is the second order Taylor expansion of $I(x)$ around $\xmap$, and it falls off as $\abs{x} \to \infty$ because it is integrable. Its mass is thus concentrated in an intermediate distance from $\xmap$. This can be seen, e.g., in \cref{arctan_densities}.
We exploit this structure by splitting up the integral in \eqref{dTV_fund_int} and bounding the integrand on a $\Sigma$-norm ball
\[
	U(r_0) := \{ x \in \Rd: \norm{x - \xmap}_\Sigma \le r_0 \}
\]
around the MAP estimate $\xmap$ and on the remaining space $\R^d \setminus U(r_0)$ separately.
On $U(r_0)$, we then control the integrand by imposing uniform bounds on the third order differentials of the log-likelihood and the log-prior density.
Outside of $U(r_0)$, we control it by imposing a quadratic lower bound on $I$.

We make the following assumptions on $\Phi$, $R$, $I$, $\xmap$, and $\Sigma$, which will be further discussed in \cref{rem:assumptions}.
\begin{assumption}
	\label{unique_minimizer}
	We have $\Phi, R \in C^3(\Rd,\R)$, $I$ has a unique global minimizer $\hat{x} = \hat{x}(y) \in \Rd$ and $D^2I(\hat{x})$ is positive definite.
\end{assumption}

\begin{assumption}
	\label{bound_third_diff}
	There exists a constant $\derbound>0$ such that
	\begin{equation*}
		\max \left\{\norm{D^3 \Phi(x)}_\Sigma, \norm{D^3 R(x)}_\Sigma\right\} \leq \derbound
	\end{equation*}
	for all $x \in \R^d$, where
	\begin{equation*}
		\norm{D^3\Phi(x)}_\Sigma := \sup \Big\{ \big\lvert D^3\Phi(x)(h_1,h_2,h_3)\big\rvert: \norm{h_1}_\Sigma, \norm{h_2}_\Sigma, \norm{h_3}_\Sigma \le 1 \Big\}.
	\end{equation*}
\end{assumption}

\begin{assumption}
	\label{quadratic_bound_I}
	There exists $0 < \delta \leq 1$ such that
	\begin{equation*}
		I(x) \ge \frac{\delta}{2}\norm{x - \xmap}_\Sigma^2 \quad \text{for all }x \in \R^d.
	\end{equation*}
\end{assumption}
Let $\Gamma(z)$ denote the classical gamma function and $\gamma(a,z)$ the lower incomplete gamma function. Then,
\[
	\Xi_d(t) := \frac{\gamma\left(\frac{d}{2}, \frac{t}{2}\right)}{\Gamma\left(\frac{d}{2}\right)} \quad \text{for all }t \ge 0, d > 0,
\]
describes the probability of a Euclidean ball in $\R^d$ with radius $\sqrt{t}$ around $0$ under a standard Gaussian measure (see \cref{tail_prob}).

The main result of this section is the following error estimate.
\begin{theorem}
	\label{main_estimate}
	Suppose that \cref{unique_minimizer,bound_third_diff,quadratic_bound_I} hold.
	Then we have
	\begin{equation}
		\label{overall_est}
		\dTV(\mu^y,\Lmuy) \le E_1(r_0) + E_2(r_0)
	\end{equation}
	for all $r_0 \ge 0$, where
	\begin{align}
		\label{eq:defM1}
		E_1(r_0) &:= c_d \epsilon^{-\frac d2} \int_0^{r_0} f(r) r^{d-1} \di r, \\
		\label{eq:defM2}
		E_2(r_0) &:= \delta^{-\frac{d}{2}}\left(1 - \Xi_d\bigg(\frac{\delta r_0^2}{\epsilon}\bigg)\right)
	\end{align}
	for all $r_0 \ge 0$,
	\begin{equation}
		\label{eq:deff}
		f(r) := \left[\exp\left(\frac{(1 + \epsilon)\derbound}{6\epsilon}r^3\right) - 1\right] \exp\left(-\frac{1}{2\epsilon}r^2\right)
	\end{equation}
	for all $r \ge 0$, and
	\begin{equation*}
		c_d := \frac{2^{1-\frac d2}}{\Gamma\left(\frac d2\right)}.
	\end{equation*}
\end{theorem}

\begin{remark}
	\label{asymp_E12}
	The two functions $E_1$ and $E_2$ are continuous and monotonic with the following asymptotic behavior. The first error term $E_1(r_0)$ obeys
	\[ E_1(0) = 0, \quad \lim_{r_0 \to \infty} E_1(r_0) = \infty, \]
	whereas the second error term $E_2(r_0)$ satisfies
	\[ E_2(0) = 2\delta^{-\frac{d}{2}}, \quad \lim_{r_0 \to \infty} E_2(r_0) = 0. \]
	This can be seen as follows.

	The function $f(r)r^{d-1}$ is bounded on the interval $[0,1]$, so that the integral $\int_0^{r_0} f(r)r^{d-1} \di r$ converges to $0$ as $r_0 \to 0$, and hence also $E_1(r_0)$. On the other hand, $f(r)r^{d-1}$ converges to $\infty$ as $r \to \infty$, so that the integral $\int_0^{r_0} f(r)r^{d-1} \di r$ and $E_1(r_0)$ converge to $\infty$ as $r \to \infty$. 
	Since $f(r)r^{d-1}$ is positive for all $r \ge 0$, $E_1$ moreover increases monotonically.
	By definition of the lower incomplete gamma function, $\Xi_d$ increases monotonically and $\Xi_d(t) \in [0,1]$ for all $t \ge 0$ and $d > 0$. Moreover, $\Xi_d(t) \to 0$ as $t \to 0$ and $\Xi_d(t) \to 1$ as $t \to \infty$.
	Consequently, $E_2(r_0)$ converges toward $2\delta^{-d/2}$ as $r_0 \to 0$, and toward $0$ as $r_0 \to \infty$.	
	The asymptotic behavior of $E_2$ is described more precisely in \cref{exp_tail_est}.	
\end{remark}

The following three propositions formalize the ideas described in the beginning of this section and constitute the prove of \cref{main_estimate}. 
\begin{proposition}[Fundamental estimate]
	\label{dTV_fund_est}
	The Laplace approximation $\Lmuy$ of $\mu^y$ satisfies
	\begin{equation*}
		\dTV(\mu^y,\Lmuy) \le \int_{\R^d} \bigabs{\exp\left(-\frac{1}{\varepsilon}R_2(x)\right) - 1} \Lmuy(\di x),
	\end{equation*}
	where $R_2(x) := I(x) - \frac12\norm{x - \xmap}_\Sigma^2$ for all $x \in \R^d$.
\end{proposition}

\begin{proof}
For a fixed $\varepsilon > 0$ we can estimate
\begin{align*}
	2\dTV(\mu^y,\Lmuy) &= \int_{\R^d} \left| \frac{1}{Z}\exp\left(-\frac{1}{\varepsilon}I(x)\right) - \frac{1}{\widetilde{Z}}\exp\left(-\frac{1}{2\varepsilon}\norm{x - \xmap}_\Sigma^2\right) \right| \di x \\
	&= \frac{1}{\widetilde{Z}} \int_{\R^d} \left| \frac{\widetilde{Z}}{Z}\exp\left(-\frac{1}{\varepsilon}I(x)\right) - \exp\left(-\frac{1}{2\varepsilon}\norm{x - \xmap}_\Sigma^2\right) \right| \di x \\
	&\le \frac{1}{\widetilde{Z}} \left| \frac{\widetilde{Z}}{Z} - 1 \right| \int_{\R^d} \exp\left(-\frac{1}{\varepsilon}I(x)\right) \di x \\
	&\quad+ \frac{1}{\widetilde{Z}} \int_{\R^d} \left| \exp\left(-\frac{1}{\varepsilon}I(x)\right) - \exp\left(-\frac{1}{2\varepsilon}\norm{x - \xmap}_\Sigma^2\right) \right| \di x \\
	&= \frac{1}{\widetilde{Z}} \left| \widetilde{Z} - Z \right|
	+ \int_{\R^d} \left| \exp\left(-\frac{1}{\varepsilon}I(x) + \frac{1}{2\varepsilon}\norm{x - \xmap}_\Sigma^2\right) - 1 \right| \Lmuy(\di x). \\
	&= \frac{1}{\widetilde{Z}} \left| \widetilde{Z} - Z \right|
	+ \int_{\R^d} \left| \exp\left(-\frac{1}{\varepsilon}R_2(x)\right) - 1 \right| \Lmuy(\di x).
\end{align*}
Now, the estimate
\begin{align*}
\left|Z - \widetilde{Z}\right|
&\le \int_{\R^d} \left| \exp\left(-\frac{1}{\varepsilon}I(x)\right) - \exp\left(-\frac{1}{2\varepsilon}\norm{x - \xmap}_\Sigma^2\right) \right| \di x \\
&= \widetilde{Z} \int_{\R^d} \left|\exp\left(-\frac{1}{\varepsilon}R_2(x)\right) - 1\right| \Lmuy(\di x)
\end{align*}
yields the proposition.
\end{proof}

\begin{proposition}[Close range estimate]
	\label{close_range_est}
	Suppose that \cref{bound_third_diff} holds.
	Then it follows that
	\begin{equation*}
		\int_{U(r_0)} \bigabs{\exp\left(-\frac{1}\epsilon R_2(x)\right) - 1} \Lmuy(\di x) \le c_d \epsilon^{-\frac d2} \int_0^{r_0} f(r) r^{d-1} \di r
	\end{equation*}
	for all $r_0 \ge 0$, where $f$ and $c_d$ are defined as in \cref{main_estimate}.
\end{proposition}

\begin{proposition}[Far range estimate]
	\label{far_range_est}
	Suppose that \cref{quadratic_bound_I} holds. Then we have
	\begin{equation}
		\label{eq:far_range_est}
		\int_{\R^d \setminus U(r_0)} \bigabs{\exp\left(-\frac{1}{\epsilon}R_2(x)\right) - 1} \Lmuy(\di x) \le 
		\delta^{-\frac{d}{2}}\left(1 - \Xi_d\bigg(\frac{\delta r_0^2}{\epsilon}\bigg)\right),
	\end{equation}
	for all $r_0 \ge 0$.
\end{proposition}

The proof of \cref{main_estimate} is now very short.
\begin{proof}[Proof of \cref{main_estimate}]
	By \cref{dTV_fund_est} we have
	\begin{equation*}
		\dTV(\mu^y,\Lmuy) \le \int_{\R^d} \bigabs{\exp\left(-\frac{1}{\epsilon}R_2(x)\right) - 1} \Lmuy(\di x).
	\end{equation*}
	Now, splitting up this integral into integrals over $U(r_0)$ and its complement and applying \cref{close_range_est,far_range_est} proves the statement.
\end{proof}

\begin{remark}
\label{rem:assumptions}
\begin{enumerate}
	\item Because of $I(\xmap) = 0$ and the necessary optimality condition $DI(\xmap) = 0$, the function $R_2(x) = I(x) - \frac12\norm{x - \xmap}_\Sigma^2$ defined in \cref{dTV_fund_est} is precisely the remainder of the second order Taylor polynomial of $I$ around $\xmap$. In \cref{close_range_est}, \cref{bound_third_diff} is used to control $R_2$ near the MAP estimate by bounding the third order differential of $I$. In \cref{far_range_est}, in turn, \cref{quadratic_bound_I} is used to control $R_2$ at a distance from $\xmap$ by bounding it from below by $-\frac{1 - \delta}{2}\norm{x - \xmap}_\Sigma^2$.
	
	\item The constant $\derbound \ge 0$ in \cref{bound_third_diff} quantifies the non-Gaussianity of the likelihood and the prior distribution and can be arbitarily large.
	\Cref{quadratic_bound_I} bounds the unnormalized log-posterior density from above by a multiple of the unnormalized log-density of the Laplace distribution, where the constant $\delta > 0$ represents the scaling factor and can be arbitrarily small. This restricts our results to posterior distributions whose tail does not decay slower than that of a Gaussian distribution.
	\Cref{quadratic_bound_I} can for example be violated if a prior distribution with heavier than Gaussian tail is chosen such as a Cauchy distribution and if the forward mapping is linear but singular.
	Our main interest lies on inverse problems with a posterior distribution that is not too different from a Gaussian distribution, since this is a setting in which the Laplace approximation can be expected to yield reasonable results.
	
	\item In case of a linear inverse problem and a Gaussian prior distribution, the Laplace approximation is exact, so that \cref{bound_third_diff,quadratic_bound_I} are trivially satisfied with $\derbound = 0$ and $\delta = 1$.
	We will see in \cref{sec:small_nonlin_limit} that \cref{bound_third_diff,quadratic_bound_I} are satisfied for nonlinear inverse problems with $\delta$ and $\derbound$ as given in \cref{Gt_delta,Gt_K} if the prior distribution is Gaussian and the nonlinearity of the forward mapping is small enough. In this case, the quadratic lower bound on $I$ in \cref{quadratic_bound_I} restricts the nonlinearity of the forward mapping to be small enough such that the tail of the posterior distribution does not decay slower than that of a Gaussian distribution.
	
	\item Note that neither in \cref{sec:estimate} nor in \cref{sec:param_free_est} we make use of the Gaussianity of the noise. Therefore, the results of these sections remain valid for non-Gaussian noise as long as the log-likelihood satisfies \cref{bound_third_diff,quadratic_bound_I}.
	In case of noise with a log-density $-\nu \in C^3(\Rd)$, the negative log-likelihood takes the form $\Phi(x) = \nu(y - G(x))$ and we have $I(x) = \nu(y - G(x)) + \epsilon R(x) + c$. Consider for example standard multivariate Cauchy noise, where
	\begin{equation*}
		\nu(\eta) = - \ln \left[ \frac{C}{\left(1 + \eunorm{\eta}^2\right)^\frac{d + 1}{2}} \right]
		= \frac{d + 1}{2} \ln \left(1 + \eunorm{\eta}^2\right) - \ln C
	\end{equation*}
	for all $\eta \in \Rd$.
	The derivatives of up to third order of $s \mapsto \ln (1 + s)$, $s \ge 0$, are bounded since they are continuous and converge to $0$ as $s$ tends to infinity. By the smoothness of $x \mapsto \eunorm{x}^2$, $\nu$ is therefore in $C^3(\Rd)$ and 
	\[ \norm{D^3\nu(x)} := \sup_{\eunorm{h_1}, \eunorm{h_2}, \eunorm{h_3} \le 1} \bigabs{D^3\nu(x)(h_1,h_2,h_3)} \] 
	is uniformly bounded.
	In case of a linear forward mapping, the uniform boundedness transfers to $\norm{D^3 \Phi(x)}$ and we can estimate 
	\[ \norm{D^3\Phi(x)}_\Sigma \le \bignorm{\Sigma^{-\frac12}}^{-3}\norm{D^3\Phi(x)}\]
	for any symmetric positive definite matrix $\Sigma$.
	
	\item We make \cref{bound_third_diff,quadratic_bound_I} globally, i.e., for all $x \in \Rd$, for the sake of simplicity. For a given $r_0 \ge 0$, \cref{main_estimate} remains valid if \cref{bound_third_diff} only holds for $\norm{x - \xmap}_\Sigma \le r_0$ and if \cref{quadratic_bound_I} only holds for $\norm{x - \xmap}_\Sigma \ge r_0$.
	This allows for prior distributions which are not supported on the whole space $\Rd$, as long as the support of the prior contains the set $U(r_0)$ and $R \in C^3(U(r_0))$. In this case, $R$ and $I$ are allowed to take values in $\overline{\R} := \R \cup \{\infty\}$ and we follow the convention $\exp(-\infty) = 0$.
	
	\item The constant $\derbound$ in \cref{bound_third_diff} can be replaced by a radial bound $\rho(\norm{x - \xmap}_\Sigma)$ with a monotonically increasing function $\rho$.
	This way, an estimate of the form \eqref{overall_est} can be obtained with $f$ replaced by
	\[
		\widetilde{f}(r) = \left(\exp\left(\frac{1 + \epsilon}{6\epsilon}\rho(r)r^3\right) - 1\right) \exp\left(-\frac{1}{2\epsilon}r^2\right).
	\]
\end{enumerate}
\end{remark}

\begin{remark}
Both the unnormalized posterior density $\exp(-\frac{1}{\epsilon}I(x))$ and the unnormalized Gaussian density $\exp(-\frac{1}{2\epsilon}\norm{x - \xmap}_\Sigma^2)$ attain their maximum $1$ in $\xmap$. The densities of $\mu^y$ and $\Lmuy$ themselves, however, take the values $1/Z$ and $1/\widetilde{Z}$ in $\xmap$ due to the different normalization, see \cref{arctan_densities}. For this reason, the probability of small balls around $\xmap$ under $\mu^y$ and $\Lmuy$ differs asymptotically by a factor of $\widetilde{Z}/Z$. This has several consequences in case that the normalization constants $Z$ and $\widetilde{Z}$ differ considerably.

One the one hand, credible regions around $\xmap$ may have considerably different size under the posterior distribution and its Laplace approximation.
On the other hand, the integrand
\[
	\frac12 \bigabs{\frac{\exp(-\frac{1}{\epsilon}I(x))}{Z} - \frac{\exp(-\frac{1}{2\varepsilon}\norm{x - \xmap}_\Sigma^2)}{\widetilde{Z}}}
\]
of the total variation distance $\dTV(\mu^y,\Lmuy)$ may, unlike the integrand of the fundamental estimate \eqref{dTV_fund_int}, have a significant amount of mass around $\xmap$, see \cref{arctan_densities}.
This means that a significant portion of the error when approximating the probability of an event under $\mu^y$ by that under $\Lmuy$ may be due to the difference in their densities \emph{near} the MAP estimate $\xmap$. So although the Laplace approximation is defined by the local properties of the posterior distribution in the MAP estimate, it is not necessarily a good local approximation around it.

A large difference in the normalization constants $Z$ and $\widetilde{Z}$ as mentioned above reflects that the log-posterior density cannot be approximated well globally by its second order Taylor polynomial around $\xmap$. 
In the proof of \cref{dTV_fund_est}, we saw that the difference in normalization is in fact bounded by the total variation of the unnormalized densities. The value of \cref{dTV_fund_est} lies in providing an estimate for the total variation error that only involves unnormalized densities.
\end{remark}

\begin{figure}[htp]
	\centering
	\begin{tikzpicture}[baseline]
		\begin{axis}[width=7.5cm, height=6cm,
			title=Probability densities,
			xmin=-6.0178637e-01, xmax=2.3982136e+00,
			ymin=0,
			thin,
			legend pos=north east,
			legend cell align=left,
			legend style={font=\small},
			legend entries={{$\mu^{y}$}, {$\Laplace_{\mu^{y}}$}},
		]
			\addplot [blue] table [x index=0, y index=1] {data/x_densities.dat};
			\addplot [red, densely dashed] table [x index=0, y index=3] {data/x_densities.dat};
		\end{axis}
	\end{tikzpicture}%
	\hskip 10pt
	\begin{tikzpicture}[baseline]
		\begin{axis}[width=7.5cm, height=6cm,
			title=Integrands,
			xmin=-6.0178637e-01, xmax=2.3982136e+00,
			ymin=0, ymax=0.14,
			thin,
			legend pos=north east,
			legend cell align=left,
			legend style={font=\small},
			legend entries={{$\dTV(\mu^y,\Laplace_{\mu^y})$}, fund.~estimate},
		]
			\addplot [green] table [x index=0, y index=1] {data/x_integrands.dat};
			\addplot [cyan, densely dashed] table [x index=0, y index=3] {data/x_integrands.dat};
		\end{axis}
	\end{tikzpicture}%
	\caption{The probability densities of a posterior distribution $\mu^y$ and its Laplace approximation $\Lmuy$ (left), as well as the integrands of the total variation distance between $\mu^y$ and $\Lmuy$ and of the fundamental estimate \eqref{dTV_fund_int} (right).}
	\label{arctan_densities}
\end{figure}
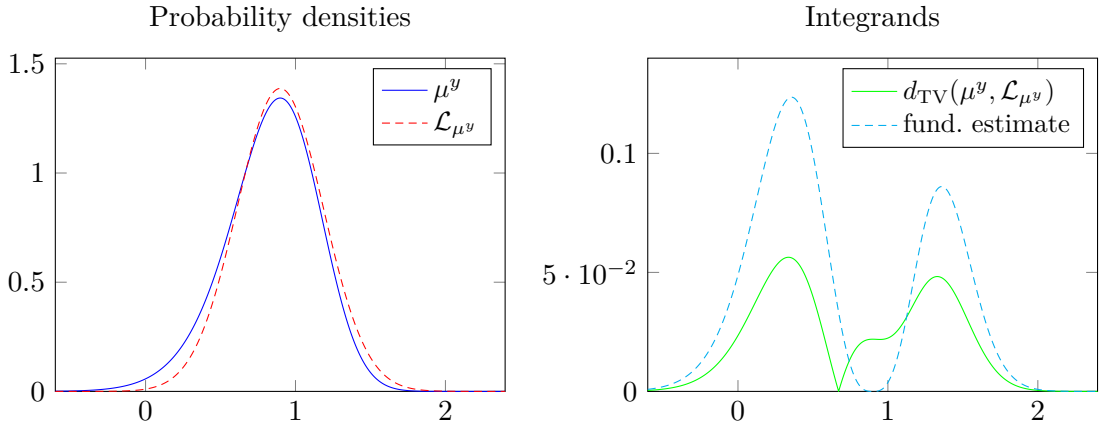

In the following subsections, we present the proofs of our close and far range estimate, and characterize the optimal choice of $r_0$.

\subsection{Proof of Proposition \ref{close_range_est}}

We consider the close range integral
\begin{equation*}
	\int_{U(r_0)} \bigabs{\exp\left(-\frac{1}\epsilon R_2(x)\right) - 1} \Lmuy(\di x)
\end{equation*}
over the $\Sigma$-norm ball with radius $r_0 \ge 0$.
The proof of our close range estimate is based upon the following estimate for the remainder term $R_2(x)$.
\begin{lemma}
	\label{R2_third_diff_est}
	If \cref{bound_third_diff} holds, then we have
	\begin{equation}
		\label{eq:R2_ineq}
		\abs{R_2(x)} \le \frac{1 + \epsilon}{6} K \norm{x-\xmap}_\Sigma^3 \quad \text{for all}~x \in \Rd.
	\end{equation}
\end{lemma}

\begin{proof}
We set $h := x - \xmap$ and write the remainder of the second order Taylor polynomial of $\Phi$ in mean-value form as
\begin{equation*}
	R_{2,\Phi}(x) := \Phi(x) - \Phi(\xmap) - D\Phi(\xmap)(h) - \frac12 D^2\Phi(\xmap)(h,h)
	= \frac16 D^3\Phi(z)(h,h,h)
\end{equation*}
for some $z \in \xmap + [0,1]h$. Since $\xmap + [0,1]h \subset U(\norm{x - \xmap}_\Sigma)$, we can now use the multilinearity of $D^3\Phi(z)$ to express $R_{2,\Phi}$ as
\begin{equation*}
	R_{2,\Phi}(x) = \frac16 D^3\Phi(z)\left(\frac{h}{\norm{h}_\Sigma},\frac{h}{\norm{h}_\Sigma},\frac{h}{\norm{h}_\Sigma}\right) \norm{h}_\Sigma^3
\end{equation*}
for some $z \in U(\norm{x - \xmap}_\Sigma)$, and estimate
\[
	\abs{R_{2,\Phi}(x)} \le \frac16 \norm{D^3\Phi(z)}_\Sigma \norm{h}_\Sigma^3
	\le \frac16 K \norm{h}_\Sigma^3
\]
using \cref{bound_third_diff}.
By proceeding similarly for $R$, we now obtain
\begin{equation*}
	\abs{R_2(x)} \le \abs{R_{2,\Phi}(x)} + \epsilon\abs{R_{2,R}(x)} 
	\le \frac16 K \norm{h}_\Sigma^3 + \frac{\varepsilon}{6} K \norm{h}_\Sigma^3.
	\qedhere
\end{equation*}
\end{proof}

Now, we can prove our close range estimate.
\begin{proof}[Proof of \cref{close_range_est}]
By \cref{R2_third_diff_est} and \eqref{form_Z_tilde}, we have
\begin{align*}
	\int_{U(r_0)} \bigabs{\exp\left(-\frac{1}\epsilon R_2(x)\right) - 1} \Lmuy(\di x)
	&\le \frac{1}{\widetilde{Z}} \int_{U(r_0)} f(\norm{x - \xmap}_\Sigma) \di x \\
	&= \frac{1}{\widetilde{Z}} \sqrt{\det \Sigma} \int_{\{u \in \Rd: \eunorm{u} \le r_0\}} f(\eunorm{u}) \di u \\
	&= \epsilon^{-\frac{d}{2}}(2\pi)^{-\frac{d}{2}} \cdot d \kappa_d \int_0^{r_0} f(r) r^{d-1} \di r,
\end{align*}
since $\eunorm{u} = \norm{x - \xmap}_\Sigma$. Here
\[
	\kappa_d = \frac{\pi^{d/2}}{\Gamma(d/2 + 1)}
\]
denotes the \emph{volume} of the $d$-dimensional Euclidean unit ball ($d\kappa_d$ is its surface area).
Using the fundamental recurrence $\Gamma(z+1) = z\Gamma(z)$, we can write
\begin{equation*}
	d\kappa_d = \frac{2\pi^{d/2}}{\Gamma(d/2)}.
	\qedhere
\end{equation*}
\end{proof}

\subsection{Proof of Proposition \ref{far_range_est}}

Now, we consider the integral
\begin{equation*}
	\int_{\Rd \setminus U(r_0)} \bigabs{\exp\left(-\frac{1}\epsilon R_2(x)\right) - 1} \Lmuy(\di x)
\end{equation*}
over the space outside of a $\Sigma$-norm ball with radius $r_0 \ge 0$.
In the proof of our far range estimate, the following expression is used to describe the probability of $\Rd \setminus U(r_0)$ under $\Lmuy$. Let $\Gamma(a,z)$ denote the upper incomplete gamma function.
\begin{lemma}
	\label{tail_prob}
	Let $\nu = \Normal(\xmap, \delta^{-1}\epsilon\Sigma)$ with $\delta > 0$.
	Then,
	\begin{equation*}
		\nu\left(\Rd \setminus U(r_0)\right)
		= \frac{1}{\Gamma(d/2)}\Gamma\bigg(\frac{d}{2},\frac{\delta r_0^2}{2\epsilon}\bigg) \quad \text{for all }r_0 \ge 0.
	\end{equation*}
\end{lemma}
\begin{proof}
We compute the tail integral explicitly using a local parameter and polar coordinates. 
This yields
\begin{align*}
	\nu(\{\norm{x - \xmap}_\Sigma \ge r_0\}) &= \frac{\delta^{d/2}}{\epsilon^{d/2}(2\pi)^{d/2}\sqrt{\det \Sigma}} \int_{\{\norm{x - \xmap}_\Sigma \ge r_0\}} \exp\left(-\frac{\delta}{2\varepsilon}\norm{x - \xmap}_\Sigma^2\right) \di x \\
	&= \frac{\delta^{d/2}}{\epsilon^{d/2}(2\pi)^{d/2}} \int_{\{\eunorm{u} \ge r_0\}} \exp\left(-\frac{\delta}{2\epsilon}\eunorm{u}^2\right) \di u \\
	&= \frac{\delta^{d/2}}{\epsilon^{d/2}(2\pi)^{d/2}} d\kappa_d \int_{r_0}^\infty \exp\left(-\frac{\delta}{2\epsilon}r^2\right) r^{d-1} \di r.
\end{align*}
We can express this integral in terms of the upper incomplete gamma function by substituting $s = \delta r^2/2\epsilon$ (note that $r'(s) = 2^{-1/2}\epsilon^{1/2}\delta^{-1/2}s^{-1/2}$) as
\begin{equation*}
	\int_{r_0}^\infty e^{-\frac{\delta}{2\epsilon}r^2} r^{d-1} \di r 
	= \int_{\frac{\delta r_0^2}{2\epsilon}}^\infty e^{-s} 2^{\frac{d}{2}-1}\epsilon^{\frac{d}{2}}\delta^{-\frac{d}{2}} s^{\frac{d}{2} - 1} \di s 
	= 2^{\frac{d}{2} - 1}\varepsilon^{\frac{d}{2}}\delta^{-\frac{d}{2}} \Gamma\left(\frac{d}{2}, \frac{\delta r_0^2}{2\epsilon}\right).
\end{equation*}
This leads to
\begin{equation*}
	\nu(\{\norm{x - \xmap}_\Sigma \ge r_0\})
	= \frac{d}{2}\frac{1}{\Gamma(d/2 + 1)}\Gamma\left(\frac{d}{2}, \frac{\delta r_0^2}{2\epsilon}\right).
\end{equation*}
Now, using the fundamental recurrence $\Gamma(z + 1) = z\Gamma(z)$ completes the proof.
\end{proof}

Now, we can prove our far range estimate.
\begin{proof}[Proof of \cref{far_range_est}]
Let $x \in \R^d \setminus U(r_0)$. We distinguish between two cases.
First, consider the case that $R_2(x) \ge 0$.
For $t \ge 0$, the estimate $\abs{e^{-t} - 1} \le 1$ holds.
This implies
\[
	\bigabs{\exp\left(-\frac{1}{\varepsilon}R_2(x)\right) - 1} \le 1.
\]
Next, consider the case that $R_2(x) < 0$.
By \cref{quadratic_bound_I}, we have
\[
	R_2(x) = I(x) - \frac12\norm{x - \xmap}_\Sigma^2 \ge \frac{\delta - 1}{2}\norm{x - \xmap}_\Sigma^2
\]
for all $x \in \R^d$.
For $t \le 0$ we have $\abs{e^{-t} - 1} = e^{-t} - 1 < e^{-t}$, and thus
\[
	\bigabs{\exp\left(-\frac{1}{\varepsilon}R_2(x)\right) - 1}
	\le \exp\left(-\frac{1}{\epsilon}R_2(x)\right)
	\le \exp\left(-\frac{\delta - 1}{2\epsilon}\norm{x - \xmap}_\Sigma^2\right).
\]
Together, this shows that
\[
	\bigabs{\exp\left(-\frac{1}{\varepsilon}R_2(x)\right) - 1} \le \exp\left(-\frac{\delta - 1}{2\epsilon}\norm{x - \xmap}_\Sigma^2\right)
\]
for all $x \in \R^d$.
Now it follows that
\[
	\bigabs{\exp\left(-\frac{1}{\varepsilon}R_2(x)\right) - 1} \exp\left(-\frac{1}{2\epsilon}\norm{x - \xmap}_\Sigma^2\right) \le \exp\left(-\frac{\delta}{2\epsilon}\norm{x - \xmap}_\Sigma^2\right)
\]
for all $x \in \R^d$.
This yields
\begin{align*}
	&\int_{\R^d \setminus U(r_0)} \bigabs{\exp\left(-\frac{1}{\epsilon}R_2(x)\right) - 1} \Lmuy(\di x) \\
	&\le \frac{1}{(2\pi\epsilon)^{d/2}\sqrt{\det \Sigma}} \int_{\R^d \setminus U(r_0)} \exp\left(-\frac{\delta}{2\epsilon}\norm{x - \xmap}_\Sigma^2\right) \di x
\end{align*}
Now, the proposition follows from
\begin{align*}
	&\frac{\delta^{d/2}}{(2\pi\epsilon)^{d/2}\sqrt{\det \Sigma}} \int_{\R^d \setminus U(r)} \exp\left(-\frac{\delta}{2\epsilon}\norm{x - \xmap}_\Sigma^2\right) \di x \\
	&= \frac{1}{\Gamma(d/2)}\Gamma\bigg(\frac{d}{2}, \frac{\delta r^2}{2\epsilon}\bigg)
	= 1 - \Xi_d\bigg(\frac{\delta r_0^2}{\epsilon}\bigg),
\end{align*}
which in turn holds by \cref{tail_prob} and the identity $\Gamma(a,z) = \Gamma(a) - \gamma(a,z)$.
\end{proof}

\subsection{Optimal choice of the parameter}

We have the following necessary optimality condition for the parameter $r_0$ in \cref{main_estimate}.

\begin{proposition}
\label{prop:optimality}
The optimal choice of $r_0$ in the error bound \eqref{overall_est} is either $0$ or satisfies
\begin{equation*}
	\exp\left(\frac{(1 + \epsilon)\derbound}{6\epsilon} r_0^3\right) - 1
	- \exp\left(\frac{1 - \delta}{2\epsilon}r_0^2\right)  = 0.
\end{equation*}
\end{proposition}

\begin{proof}
The terms $E_1$ and $E_2$ are differentiable on $[0,\infty)$. Clearly, the optimal $r_0$ is either $0$ or satisfies the identity
\begin{equation}
	\label{eq:Eidentity}
	E'(r_0) = E_1'(r_0) + E_2'(r_0) = 0.
\end{equation}
We have that
\begin{equation*}
	E_1'(r_0) = c_d \epsilon^{-\frac d2}f(r_0) r_0^{d-1} = c_d \epsilon^{-\frac d2}\left(\exp\left(\frac{(1 + \epsilon)\derbound}{6\epsilon} r_0^3\right) - 1\right) \exp\left(-\frac{1}{2\epsilon} r_0^2\right) r_0^{d-1}
\end{equation*}
and
\begin{equation*}
	E_2'(r_0) = -\frac{2\delta^{1-\frac d2} r_0}{\epsilon} \Xi_d'\left(\frac{\delta r_0^2}\epsilon \right)
	= -\frac{2}{2^{d/2}\Gamma(d/2)} \cdot \frac{r_0^{d-1}}{\epsilon^{d/2}} \exp\left(-\frac{\delta}{2\epsilon}r_0^2\right)
\end{equation*}
Identity \eqref{eq:Eidentity} now corresponds to
\begin{equation*}
	\exp\left(\frac{(1 + \epsilon)\derbound}{6\epsilon} r_0^3\right) - 1
	- \exp\left(\frac{1 - \delta}{2\epsilon}r_0^2\right) = 0
\end{equation*}
which yields the result.
\end{proof}

\begin{remark}
The right hand side of the far range estimate \eqref{eq:far_range_est} can be written as
\[
	c_d \epsilon^{-\frac{d}{2}} \int_{r_0}^\infty \exp \left(-\frac{\delta}{2\epsilon}r^2\right) r^{d-1} \di r,
\]
where $c_d$ is defined as in \cref{main_estimate}. The optimal choice of $r_0$ is therefore one for which the integrands $f(r_0)r^{d-1}$ and $\exp(-\delta r_0^2/2\epsilon)r^{d-1}$ of the close and the far range estimate take the same value.
\end{remark}


\section{Explicit error estimate}
\label{sec:param_free_est}

Here, we present a non-asymptotic error estimate in terms of $\derbound$, $\delta$, $\epsilon$, and the problem dimension $d$.
While \cref{main_estimate} constitutes a non-asymptotic error estimate and is the sharpest of our three main results, it is not immediately clear how the non-Gaussianity of the likelihood and the prior distribution, as quantified by the constant $\derbound$ in \cref{bound_third_diff}, the noise level, and the problem dimension affect the error bound. The purpose of the following theorem is to make this influence more explicit.
\begin{theorem}
	\label{nonasymp_est_dim}
	Suppose that $I$, $\Phi$, and $R$ satisfy \cref{unique_minimizer,bound_third_diff,quadratic_bound_I}.
	If $K$, $\delta$, $\epsilon$, and $d$ satisfy
	\begin{align}
		\label{cond_E2_dim}
		\frac{6\delta^\frac32}{(1 + \epsilon)\epsilon^\frac12 K} &\geq \max \left\{8d^\frac32, \left(8 \ln \left(\frac{2}{C(1 + \epsilon)\epsilon^\frac12 K\delta^\frac{d}{2}}\cdot\frac{\Gamma\big(\frac{d}{2}\big)}{\Gamma\big(\frac{d}{2} + \frac32\big)}\right)\right)^\frac32\right\}
	\end{align}
	with $C := \sqrt{2}e/3$, then
	\begin{equation}
		\label{eq:nonasymp_est_dim}
		\dTV\left(\mu^{y}, \Laplace_{\mu^{y}}\right) \leq 2C(1 + \epsilon)\epsilon^\frac12 \derbound \frac{\Gamma\left(\frac{d}{2} + \frac32\right)}{\Gamma\left(\frac{d}{2}\right)}.
	\end{equation}
\end{theorem}

\begin{remark}
Condition \eqref{cond_E2_dim} can be interpreted in the following way. For given $d$ and $\delta$, it imposes an upper bound on the noise level $\epsilon^{1/2}$ and $K$, whereas for given $\delta$, $K$, and $\epsilon$, it imposes an upper bound on the dimension $d$.
As $d \to \infty$, the ratio $\Gamma\big(\frac{d}{2} + \frac{3}{2}\big)/\Gamma\big(\frac{d}{2}\big)$ grows in the order of $d^{3/2}$, see \cite[pp.~67--68]{Tem:1996b}.
\end{remark}

In order to prove this theorem, we introduce an exponential tail estimate for the Laplace approximation, which is a modified version of \cite[Prop.~4]{SSW:2020}.
\begin{lemma}
	\label{exp_tail_est}
	Let $\nu = \Normal(\xmap, \delta^{-1}\epsilon\Sigma)$ with $\delta > 0$.
	Then,
	\begin{equation*}
		\nu(\R^d \setminus U(r)) = 1 - \Xi_d\bigg(\frac{\delta r^2}{\epsilon}\bigg) \le 2\exp\left(-\frac{\delta r^2}{8\varepsilon}\right)
	\end{equation*}
	holds for all $r \ge 2(d\epsilon/\delta)^{1/2}$.
\end{lemma}
\begin{proof}
By \cref{tail_prob}, we have
\[
	\nu(\Rd \setminus U(r_0)) = \frac{1}{\Gamma(d/2)}\Gamma\bigg(\frac{d}{2},\frac{\delta r_0^2}{2\epsilon}\bigg) = 1 - \Xi_d\bigg(\frac{\delta r_0^2}{\epsilon}\bigg).
\]
Let $x \sim \Normal(\xmap, \delta^{-1}\varepsilon \Sigma)$. Then $u := \Sigma^{-1/2}(x - \xmap) \sim \Normal(0,\delta^{-1}\epsilon I_d)$.
The concentration inequality for Gaussian measures yields
\begin{equation*}
	\Prob(\abs{u} > s + \Exp[\abs{u}]) \le \Prob\big(\big|\abs{u} - \Exp[\abs{u}]\big| > s\big) \le 2 \exp \left(-\frac{s^2}{2\sigma^2} \right) \quad \text{for all }s > 0,
\end{equation*}
where $\sigma := \sup_{\abs{z} \le 1} \Exp[\abs{\scalprod{z,u}}^2]$, see \cite[Chapter 3]{LT:2002}.
Now, 
\begin{equation*}
	\sigma^2 = \lambda_\text{max}(\delta^{-1}\varepsilon I_d) = \delta^{-1}\varepsilon,
\end{equation*}
and
\begin{equation*}
	\Exp[\abs{u}] \le \Exp[\abs{u}^2]^\frac12 = \delta^{-\frac12}\varepsilon^\frac12\tr(I_d)^\frac12 = \delta^{-\frac12}\varepsilon^\frac12 d^\frac12.
\end{equation*}
By choosing $s = \delta^{-\frac12}\varepsilon^{1/2}d^{1/2}$ and using that $r \ge 2s$, we obtain
\[
	\mu^y(\R^d \setminus U(r)) = \Prob(\abs{u} > r)
	\le \Prob\left(\abs{u} > \frac{r}{2} + \Exp[\abs{u}]\right)
	\le 2\exp\left(-\frac{\delta r^2}{8\varepsilon}\right).
	\qedhere
\]
\end{proof}

\begin{proof}[Proof of \cref{nonasymp_est_dim}]
We choose
\[
	r_0 := \left(\frac{6\epsilon}{(1 + \epsilon)K}\right)^\frac13.
\]
According to \cref{main_estimate}, we then have 
\[
	\dTV\left(\mu^{y}, \Laplace_{\mu^{y}}\right) \leq E_1(r_0;K,\epsilon,d) + E_2(r_0;\delta,\epsilon,d).
\]
For all $t \ge 0$, the exponential function satisfies the estimate 
\[
	\exp(t) - 1 = \left(1 - \exp(-t)\right)\exp(t) \le t\exp(t).
\]
Therefore, we have
\begin{equation*}
\begin{split}
	E_1(r_0;K,\epsilon,d) &= \frac{2}{\Gamma\left(\frac{d}{2}\right)} (2\epsilon)^{-\frac{d}{2}} \int_0^{r_0}  \left(\exp\left(\frac{(1 + \epsilon)\derbound}{6\epsilon}r^3\right) - 1\right) \exp\left(-\frac{1}{2\epsilon}r^2\right) r^{d-1} \di r \\
	&\le \frac{2}{\Gamma\left(\frac{d}{2}\right)} (2\epsilon)^{-\frac{d}{2} - 1} \frac{(1 + \epsilon)\derbound}{3} \int_0^{r_0} \exp\left(\frac{(1 + \epsilon)\derbound}{6\epsilon} r^3 - \frac{1}{2\epsilon}r^2\right) r^{d+2} \di r.
\end{split}
\end{equation*}
By the choice of $r_0$, we have
\begin{equation*}
	\frac{(1 + \epsilon)\derbound}{6\epsilon} r^3 \leq 1 \quad \text{for all }r \in [0,r_0],
\end{equation*}
so that the integral is bounded by
\begin{equation*}
	e \int_0^{r_0} \exp\left(-\frac{1}{2\epsilon}r^2\right) r^{d+2} \di r.
\end{equation*}
By substituting $s = r^2/2\epsilon$, we can in turn express this integral as
\begin{equation*}
	\frac12(2\epsilon)^{\frac{d}{2} + \frac32} \int_0^{\frac{r_0^2}{2\epsilon}} e^{-s}s^{\frac{d}{2} + \frac12} \di s
	= \frac12(2\epsilon)^{\frac{d}{2} + \frac32} \gamma\left(\frac{d}{2} + \frac32, \frac{r_0^2}{2\epsilon}\right).
\end{equation*}
Now, we use the inequality $\gamma(a,z) \le \Gamma(a)$ to obtain that
\[
	E_1(r_0;K,\epsilon,d) \le \frac{2^\frac12 e}{3} (1 + \epsilon) \epsilon^\frac12 \derbound \frac{\Gamma\big(\frac{d}{2} + \frac32\big)}{\Gamma\big(\frac{d}{2}\big)}.
\]

By condition \eqref{cond_E2_dim}, we have
\[
	r_0 = \left(\frac{6\epsilon}{(1 + \epsilon)K}\right)^\frac13 \ge 2\left(\frac{d\epsilon}{\delta}\right)^\frac12.
\]
Thus, we may apply \cref{exp_tail_est}, which yields
\begin{equation*}
\begin{split}
	E_2(r_0;\delta,\epsilon,d) &= \delta^{-\frac{d}{2}} \left(1 - \Xi\bigg(-\frac{\delta r_0^2}{\epsilon}\bigg)\right)
	\le 2\delta^{-\frac{d}{2}} \exp\left(-\frac{\delta r_0^2}{8\epsilon}\right) \\
	&= 2\delta^{-\frac{d}{2}} \exp\left(-\frac18\left(\frac{6\delta^\frac32}{(1 + \epsilon)\epsilon^\frac12 K}\right)^\frac23\right)
	\le C(1 + \epsilon)\epsilon^\frac12 \derbound \frac{\Gamma\big(\frac{d}{2} + \frac32\big)}{\Gamma\big(\frac{d}{2}\big)}
\end{split}
\end{equation*}
by condition \eqref{cond_E2_dim}.
Now, we obtain by summing up that
\[
	\dTV\left(\mu^{y}, \Laplace_{\mu^{y}}\right) \le 2C (1 + \epsilon)\epsilon^\frac12 \derbound \frac{\Gamma\big(\frac{d}{2} + \frac32\big)}{\Gamma\big(\frac{d}{2}\big)}.
	\qedhere
\]
\end{proof}

\subsection{Asymptotic behavior for fixed and increasing problem dimension}

Now, we describe the convergence of the Laplace approximation for a sequence of nonlinear problems that satisfy \cref{unique_minimizer,bound_third_diff,quadratic_bound_I} with varying bounds $\{K_n\}_{n\in\N}$ and $\{\delta_n\}_{n\in\N}$, respectively, and varying squared noise levels $\{\epsilon_n\}_{n\in\N}$, both in case of a fixed and an increasing problem dimension.
We denote the data by $y_n$, the prior distribution by $R_n$, the scaled negative log-likelihood by $\Phi_n$, and set $I_n(x) = \Phi_n(x) + \epsilon_n R_n(x)$.

First, we consider the case that the problem dimension $d$ remains constant.
\begin{corollary}[Fixed problem dimension]
	\label{conv_rate_fixed_dim}
	Suppose that $I_n$, $\Phi_n$, and $R_n$ satisfy \cref{unique_minimizer,bound_third_diff,quadratic_bound_I}.
	If $\epsilon_n^{1/2}\derbound_n \to 0$ and if there exist $\underline{\delta} > 0$ and $N_0 \in \N$ such that
	\[
		\delta_n \ge \underline{\delta}
		\quad\text{and}\quad
		\epsilon_n \le 1
	\]
	for all $n \ge N_0$, then there exist $C = C(d) > 0$ and $N_1 \ge N_0$ such that
	\[
		\dTV\left(\mu^{y_n}, \Laplace_{\mu^{y_n}}\right) \leq C \epsilon_n^\frac12\derbound_n
	\]
	for all $n \ge N_1$.
\end{corollary}

\begin{proof}
Since $\{\delta_n\}_{n\in\N}$ is bounded from below and $\{\epsilon_n\}_{n\in\N}$ is bounded from above, the left hand side of \eqref{cond_E2_dim} is bounded from below by $C_1/\epsilon_n^{1/2} K_n$ for some $C_1 > 0$. On the other hand, the right hand side of \eqref{cond_E2_dim} is bounded from above by $(8 \ln C_2/\epsilon_n^{1/2} K_n)^{2/3}$ for large enough $n$ and some $C_2 > 0$, since $\{\delta_n\}_{n\in\N}$ is bounded from below and $\{\epsilon_n\}_{n\in\N}$ is bounded from below by $0$. Consequently, there exists $N_1 \ge N_0$ such that condition \eqref{cond_E2_dim} holds for all $n \ge N_1$ by the convergence $\epsilon_n^{1/2} K_n \to 0$ and since $\lim_{t \to \infty} t^{-2/3} \ln t = 0$.
Now, \cref{nonasymp_est_dim} yields the proposition.
\end{proof}

\begin{remark}
\Cref{conv_rate_fixed_dim} covers two cases of particular interest: That of $K_n \to 0$ while $\epsilon_n = \epsilon$ remains constant, which yields a rate of $K_n$, and that of $\epsilon_n \to 0$ while $\derbound_n = \derbound$ remains constant, which yields a rate of $\epsilon_n^{1/2}$.
The former case can, for example, occur if the sequence of forward mappings $G_n$ converges pointwise towards a linear mapping, see \cref{sec:small_nonlin_limit}.
The convergence rate in the latter case, i.e., in the small noise limit, agrees with the rate established in \cite[Theorem 2]{SSW:2020} if we set $\epsilon_n = \frac{1}{n}$.
\end{remark}

Now, we consider the case of an increasing problem dimension $d \to \infty$. To this end, we index $K_d$, $\delta_d$, $\epsilon_d$, and $R_d$ by $d \in \N$.
\begin{corollary}[Increasing problem dimension]
	\label{cor_inc_prob_dim}
	Suppose that $I_d$, $\Phi_d$, and $R_d$ satisfy \cref{unique_minimizer,bound_third_diff,quadratic_bound_I} and that $\epsilon_d^{1/2}\derbound_d \to 0$.
	If there exists $N_0 \in \N$ such that $\delta_d \le e^{-1/2}$, $\epsilon_d \le 1$, and
	\begin{equation}
		\label{asymp_cond_dim}
		\frac{3}{\epsilon_d^\frac12\derbound_d} \ge \left(\frac{8}{\delta_d}\ln \frac{1}{\delta_d}\right)^\frac32 d^\frac32
	\end{equation}
	for all $d \ge N_0$,
	then for every $C > 2\sqrt{2}e/3$, there exists $N_1 \ge N_0$ such that
	\[
		\dTV\left(\mu^{y_d}, \Laplace_{\mu^{y_d}}\right) \leq C \epsilon_d^\frac12\derbound_d d^\frac32
	\]
	for all $d \ge N_1$.
\end{corollary}

\begin{proof}
We can write condition \eqref{cond_E2_dim} as
\begin{align}
	\label{cond2_transformed}
	\left(\frac{6}{(1 + \epsilon_d)\epsilon_d^\frac12\derbound_d}\right)^\frac23
	&\ge \frac{4d}{\delta_d} \quad\text{and} \\
	\label{cond1_split_up_RHS}
	\left(\frac{6}{(1 + \epsilon_d)\epsilon_d^\frac12\derbound_d}\right)^\frac23
	&\ge \frac{8}{\delta_d} \ln \left(\frac{2}{C(1 + \epsilon_d)\epsilon_d^\frac12\derbound_d}\cdot\frac{\Gamma\big(\frac{d}{2}\big)}{\Gamma\big(\frac{d}{2} + \frac32\big)}\right) + \frac{4d}{\delta_d} \ln \frac{1}{\delta_d}.
\end{align}
By \cite[pp.~67--68]{Tem:1996b}, we have
\begin{equation}
	\label{asymp_Gamma_ratio}
	\lim_{d \to \infty} \frac{\Gamma\big(\frac{d}{2} + \frac32\big)}{\Gamma\big(\frac{d}{2}\big)}\left(\frac{d}{2}\right)^{-\frac32} = 1,
\end{equation}
so that the first term on the right hand side of \eqref{cond1_split_up_RHS} is bounded from above by
\[
	8e^\frac12 \ln \frac{C_1}{\epsilon_d^\frac12\derbound_d d^\frac32} \le 8e^\frac12 \ln \frac{C_1}{\epsilon_d^\frac12\derbound_d}
\]
for some $C_1 > 0$, which in turn is bounded from above by
\begin{equation}
	\label{cond1_LHS_lower_bound}
	\frac12\left(\frac{3}{\epsilon_d^{1/2}\derbound_d}\right)^\frac23 
	\le \frac12\left(\frac{6}{(1 + \epsilon_d)\epsilon_d^{1/2}\derbound_d}\right)^\frac23
\end{equation}
for large enough $d$, due to the convergence $\epsilon_d^{1/2}\derbound_d \to 0$, the boundedness from above of $\{\epsilon_n\}_{n\in\N}$, and since $\lim_{t \to \infty} t^{-2/3} \ln t = 0$.
By \eqref{asymp_cond_dim}, the second term on the right hand side of \eqref{cond1_split_up_RHS} is bounded from above by \eqref{cond1_LHS_lower_bound} for all $d \ge N_0$ as well.
Due to the assumption $\delta_d \le e^{-1/2}$ and \eqref{cond1_LHS_lower_bound}, condition \eqref{asymp_cond_dim} also ensures that \eqref{cond2_transformed} is satisfied for all $d \ge N_0$.
Therefore, condition \eqref{cond_E2_dim} is satisfied for large enough $d$.
Now, \cref{nonasymp_est_dim} and \eqref{asymp_Gamma_ratio} yield that for every $C > 2\sqrt{2}e/3$ there exists $N_1 \ge N_0$ such that
\[
	\dTV\left(\mu^{y_d}, \Laplace_{\mu^{y_d}}\right) \leq C \epsilon_d^\frac12\derbound_d d^\frac32
\]
for all $d \ge N_1$.
\end{proof}


\section{Perturbed linear problems with Gaussian prior}
\label{sec:small_nonlin_limit}

In this section we consider the case that the forward mapping $G$ is given by a linear mapping with a small nonlinear perturbation, i.e., that
\begin{equation}
	\label{lin_op_nonlin_perturb}
	G_\tau(x) = Ax + \tau F(x),
\end{equation}
with $A \in \R^{d \times d}$, $F \in C^3(\R^d)$, and $\tau\geq 0$.
We then quantify the error of the Laplace approximation for small $\tau$, that is, when the nonlinearity of the forward mapping is small, and for fixed problem dimension $d$.
In order to isolate the effect of the nonlinearity on estimate \eqref{eq:nonasymp_est_dim}, we consider the case when not only the noise, but also the prior distribution is Gaussian. This ensures that all non-Gaussianity in the posterior distribution results from the nonlinearity of $G_\tau$. 

We assign a prior distribution $\mu = \Normal(m_0,\Sigma_0)$ with $m_0 \in \R^d$ and symmetric, positive definite $\Sigma_0 \in \R^{d \times d}$, i.e., we set
\begin{align}
	R(x) &:= -\ln\left(\frac{1}{(2\pi)^{d/2}\sqrt{\det \Sigma_0}} \exp\left(-\frac{1}{2}\norm{x - m_0}_{\Sigma_0}^2\right)\right) \nonumber \\
	&= \frac{1}{2} \norm{x - m_0}_{\Sigma_0}^2 + \frac{d}{2}\ln 2\pi + \frac12\ln \det \Sigma_0. \label{Gaussian_R}
\end{align}
For each $\tau \ge 0$ we denote the data by $y_\tau$ and the scaled negative log-likelihood by $\Phi_\tau(x) = \frac12\eunorm{G_\tau(x) - y_\tau}^2$.

We make the following assumptions on the function $I_\tau$ and the perturbation $F$. Let $B(r) \subset \R^d$ denote the closed Euclidean ball with radius $r$ around the origin.
\begin{assumption}
	\label{ex_min_small_tau}
	We assume that there exists $\tau_0 > 0$ such that for all $\tau \in [0,\tau_0]$,
	\[
		I_\tau(x) = \frac12\eunorm{Ax + \tau F(x) - y_\tau}^2 + \frac{\epsilon}{2}\norm{x - m_0}_{\Sigma_0}^2 + c_\tau
	\]
	has a unique minimizer $\xmap_\tau$ with $D^2I_\tau(\xmap_\tau) > 0$.
	Furthermore, we assume that $y_\tau$, $\xmap_\tau$ and $\Sigma_\tau := D^2I_\tau(\xmap_\tau)^{-1}$ converge as $\tau \to 0$ with $\lim_{\tau \to 0} \Sigma_\tau > 0$ and denote their limits by $y$, $\xmap$, and $\Sigma$, respectively.
\end{assumption}

\begin{assumption}
	\label{assumptions_F}
	There exist constants $C_0, \dots, C_3>0$ and $\tau_0>0$ such that
	\begin{equation*}
		\norm{D^jF(x)}_{\Sigma_\tau} \le C_j, \qquad j=0,\dots,3,
	\end{equation*}
	for all $x \in \Rd$ and $\tau \in [0,\tau_0]$, and there exists $M > 0$ such that
	\begin{equation*}
		D^3 F \equiv 0 \qquad \text{on }\R^d \setminus B(M).
	\end{equation*}
\end{assumption}

The idea behind the following theorem is to make explicit how the nonlinearity of the forward mapping, as quantified by $\tau$ and the constants $C_0, \dots, C_3, M$ in \cref{assumptions_F}, influences the total variation error bound of \cref{nonasymp_est_dim}.
\begin{theorem}
	\label{conv_rate_nonlin_perturb}
	Suppose that \cref{ex_min_small_tau,assumptions_F} hold. Then, there exists $\tau_1 \in (0,\tau_0]$ such that
	\[
		\dTV\left(\mu^{y_\tau},\Laplace_{\mu^{y_\tau}}\right) \le C(1 + \epsilon)\epsilon^\frac12 \left(V(\tau)\tau + \frac{W}{2}\tau^2\right)
	\]
	for all $\tau \in [0,\tau_1]$, where
	\begin{align*}
		C  &:= \frac23\sqrt{2}e \frac{\Gamma\left(\frac{d}{2} + \frac32\right)}{\Gamma\left(\frac{d}{2}\right)}, \\
		V(\tau) &:= C_3\left(\norm{A}M + \eunorm{y_\tau}\right) + 3C_2\bignorm{A\Sigma_\tau^\frac12}, \\
		W &:= C_3 C_0 + 3 C_2 C_1.
	\end{align*}
	Moreover, $\{V(\tau)\}_{\tau \in [0,\tau_1]}$ is bounded.
\end{theorem}

\begin{remark}
\begin{enumerate}
	\item The choice of the upper bound $\tau_1$ is made explicit in the proof of \cref{conv_rate_nonlin_perturb} and depends on $d$ and $\epsilon$, i.a., through $\delta_0$ as defined in \cref{Gt_delta}.

	The proof of \cref{conv_rate_nonlin_perturb} can be adapted to yield a result analogous to \cref{cor_inc_prob_dim} in the case when the problem dimension $d$ tends to $\infty$ while the size $\tau_d$ of the perturbation tends to $0$. Then, $\delta_{\tau_d}$ may converge to $0$, and \eqref{asymp_cond_dim} imposes a bound on the rate at which $\{\tau_d\}_{d\in\N}$ tends to $0$.

\item By the boundedness and continuity of $F$, $G_\tau$ $\Gamma$-converges toward $A$.
	By the fundamental theorem of $\Gamma$-convergence and \cref{ex_min_small_tau}, this, in turn, implies that $\xmap$ is the minimizer of
	\[
		I(x) = \frac12\eunorm{Ax - y}^2 + \frac{\epsilon}{2}\norm{x - m_0}_{\Sigma_0}^2 + c
	\]
	and that $\Sigma = D^2I(\xmap)^{-1}$.

	\item \Cref{conv_rate_nonlin_perturb} remains valid if the assumption that $D^3F \equiv 0$ outside of a bounded set is replaced by
	\[
		\norm{D^3F(x)}_{\Sigma_\tau} \le C_3 M \eunorm{x}^{-1}
	\]
	for all $x \in \R \setminus \{0\}$ and $\tau \in [0,\tau_0]$.
\end{enumerate}
\end{remark}

In order to prove \cref{conv_rate_nonlin_perturb}, we first show that \cref{bound_third_diff,quadratic_bound_I} are satisfied for small enough $\tau$ and determine the bounds $\derbound_\tau$ and $\delta_\tau$.
Then, we derive the error estimate for the perturbed linear case.

\subsection{Verifying Assumption \ref{quadratic_bound_I}}

We verify that \cref{quadratic_bound_I} holds for small enough $\tau$ and determine $\delta_\tau$.
First, we estimate $I_\tau$ from below.
\begin{lemma}
	\label{It_est}
	For all $\tau \ge 0$ and $x \in \Rd$, we have
	\begin{align*}
		I_\tau(x) &\ge \frac12\eunorm{G_\tau(x) - G_\tau(\xmap_\tau)}^2 + \frac{\epsilon}{2}\norm{x - \xmap_\tau}_{\Sigma_0}^2 \\
		&\quad - \eunorm{G_\tau(x) - G_\tau(\xmap_\tau) - DG_\tau(\xmap_\tau)(x -\xmap_\tau)}\cdot\eunorm{G_\tau(\xmap_\tau) - y_\tau}.
	\end{align*}
\end{lemma}
\begin{proof}
Since $\xmap_\tau$ satisfies the necessary optimality condition 
\[
	D\Phi_\tau(\xmap_\tau) + \epsilon DR(\xmap_\tau) = DI_\tau(\xmap_\tau) = 0, 
\]
we can write $I_\tau(x)$ as
\begin{equation*}
	I_\tau(x) = \Phi_\tau(x) - \Phi_\tau(\xmap_\tau) - D\Phi_\tau(\xmap_\tau)(x - \xmap_\tau) + \epsilon\Big(R(x) - R(\xmap_\tau) - DR(\xmap_\tau)(x - \xmap_\tau)\Big)
\end{equation*}
for all $x \in \R^d$.
For the log-likelihood, we have
\begin{equation*}
	\Phi_\tau(x) - \Phi_\tau(\xmap_\tau) = \frac12 |G_\tau(x) - G_\tau(\xmap_\tau)|^2 + \scalprod{G_\tau(x) - G_\tau(\xmap_\tau), G_\tau(\xmap_\tau) - y_\tau},
\end{equation*}
and
\[
	D\Phi_\tau(\xmap_\tau)(x - \xmap_\tau) = \scalprod{DG_\tau(\xmap_\tau)(x - \xmap_\tau), G_\tau(\xmap_\tau) - y_\tau}
\]
for all $x \in \Rd$.
From this, we obtain
\begin{equation}
\label{skewed_Phi_t}
\begin{split}
	& \Phi_\tau(x) - \Phi_\tau(\xmap_\tau) - D\Phi_\tau(\xmap_\tau)(x - \xmap_\tau) \\
	& \ge \frac12\eunorm{G_\tau(x) - G_\tau(\xmap_\tau)}^2 - \eunorm{G_\tau(x) - G_t(\xmap_\tau) - DG_\tau(\xmap_t)(x -\xmap_\tau)}\cdot\eunorm{G_\tau(\xmap_\tau) - y_\tau} 
\end{split}
\end{equation}
using the Cauchy--Schwarz inequality.
For the log-prior density, we have
\begin{equation*}
	DR(\xmap_\tau)(x - \xmap_\tau) = \scalprod{\xmap_\tau - m_0,x - m_0}_{\Sigma_0} - \norm{\xmap_\tau - m_0}_{\Sigma_0}^2
\end{equation*}
for all $x \in \Rd$, and thus
\begin{equation}
	R(x) - R(\xmap_\tau) - DR(\xmap_\tau)(x - \xmap_\tau) = \frac12\norm{x - \xmap_\tau}_{\Sigma_0}^2. \label{skewed_R}
\end{equation}
Now, adding up \eqref{skewed_Phi_t} and \eqref{skewed_R} multiplied with $\epsilon$ yields the proposition.
\end{proof}

\begin{proposition}
	\label{Gt_delta}
	Suppose that \cref{assumptions_F} holds.
	Then there exists $\tau_0 > 0$, such that $I_\tau$ satisfies 
	\begin{equation*}
		I_\tau(x) \geq \frac{\delta_\tau}{2} \norm{x-\xmap_\tau}_{\Sigma_\tau}^2
\end{equation*}		
	for all $x\in \R^d$ and $\tau \in [0,\tau_0]$, where
	\[
		\delta_\tau := \gamma_1(\tau) - \gamma_2(\tau)\eunorm{G_\tau(\xmap_\tau) - y_\tau} > 0
	\]
	with
	\[
		\gamma_1(\tau) := \frac{1}{\bignorm{\Sigma_\tau^{-\frac12}(A^\T A + \epsilon \Sigma_0^{-1})^{-\frac 12}}^2} -  C_1^2\tau^2
		\quad\text{and}\quad
		\gamma_2(\tau) :=  C_2\tau
	\]
	for all $\tau \in [0,\tau_0]$. Furthermore, $\lim_{\tau \to 0} \delta_\tau > 0$.
\end{proposition}
\begin{proof}
Let $x \in \Rd$ be arbitrary.
Then, there exists $z_1 \in \R$ such that $F(x) - F(\xmap_\tau) = DF(z_1)(x - \xmap_\tau)$. 
Therefore,
\[
	\eunorm{F(x) - F(\xmap_\tau)} \le \norm{DF(z_1)}_{\Sigma_\tau} \norm{x - \xmap_\tau}_{\Sigma_\tau} \le C_1 \norm{x - \xmap_\tau}
\]
by \cref{assumptions_F}.
Moreover, we have
\[
	\norm{x - \xmap_\tau}_{\Sigma^\tau} \le \bignorm{\Sigma_\tau^{-\frac 12} \left(A^\T A + \epsilon \Sigma_0^{-1}\right)^{-\frac 12}} \cdot \bigeunorm{\left(A^\T A + \epsilon\Sigma_0^{-1}\right)^\frac12 (x - \xmap_\tau)},
\]
and hence
\begin{align*}
	&\frac 12 \eunorm{G_\tau(x) - G_\tau(\xmap_\tau)}^2 + \frac\epsilon 2 \norm{x-\xmap_\tau}_{\Sigma_0}^2 \\
	& \geq \frac 12\eunorm{A(x-\xmap_\tau)}^2 - \frac{\tau^2}{2} \eunorm{F(x) - F(\xmap_\tau)}^2 + \frac\epsilon 2 \norm{x-\xmap_\tau}_{\Sigma_0}^2 \\
	& \geq \frac 12 \bigscalprod{\left(A^\T A + \epsilon \Sigma_0^{-1}\right) (x-\xmap_\tau), x-\xmap_\tau} - \frac{\tau^2C_1^2}2 \norm{x - \xmap_\tau}_{\Sigma_\tau}^2 \\
	&\geq \frac12\gamma_1(\tau) \norm{x - \xmap_\tau}_{\Sigma_\tau}^2
\end{align*}
for all $\tau \le \tau_0$.
There also exists $z_2 \in \Rd$ such that
\begin{equation*}
	G_\tau(x) - G_\tau(\xmap_\tau) - DG_\tau(\xmap_\tau)(x - \xmap_\tau) = \frac12 D^2G_\tau(z_2)(x - \xmap_\tau, x - \xmap_\tau),
\end{equation*}
and $D^2G_\tau = \tau D^2F$.
By \cref{assumptions_F}, we thus have
\begin{equation*}
\begin{split}
	\eunorm{G_\tau(x) - G_\tau(\xmap_\tau) - DG_\tau(\xmap_\tau)(x - \xmap_\tau)} 
	&\le \frac{\tau}{2}\norm{D^2 F(z_2)}_{\Sigma_\tau} \norm{x - \xmap_\tau}_{\Sigma_\tau}^2 \\
	&\le \frac{\tau C_2}{2} \norm{x - \xmap_\tau}_{\Sigma_\tau}^2
	= \frac12\gamma_2(\tau) \norm{x - \xmap_\tau}_{\Sigma_\tau}^2
\end{split}
\end{equation*}
for all $\tau \leq \tau_0$.
Now, \cref{It_est} yields that
\[
	I_\tau(x) \ge \frac12\gamma_1(\tau) \norm{x - \xmap_\tau}_{\Sigma_\tau}^2 - \frac12\gamma_2(\tau) \norm{x - \xmap_\tau}_{\Sigma_\tau}^2 \eunorm{G_\tau(\xmap_\tau) - y_\tau} = \frac{\delta_\tau}{2} \norm{x - \xmap_\tau}_{\Sigma_\tau}^2.
\]

It remains to show that $\lim_{\tau\to 0} \delta_\tau > 0$. 
The convergence of $y_\tau$ and $\xmap_\tau$ yields
\[
	G_\tau(\xmap_\tau) - y_\tau \to A\xmap - y.
\]
Now, it follows from $\lim_{\tau \to 0} \gamma_2(\tau) = 0$ and the convergence of $\Sigma_\tau$ that
\[
	\lim_{\tau \to 0} \delta_t = \lim_{\tau \to 0} \gamma_1(\tau) = \frac{1}{\bignorm{\Sigma^{-\frac 12} (A^\T A + \epsilon \Sigma_0^{-1})^{-\frac 12}}^2} > 0.
	\qedhere
\]
\end{proof}

\subsection{Verifying Assumption \ref{bound_third_diff}}

Now, we verify that \cref{bound_third_diff} holds for small $\tau$. The following proposition also describes how the nonlinearity of the forward mapping translates into non-Gaussianity of the likelihood, as quantified by the costant $\derbound_d$.
\begin{proposition}
	\label{Gt_K}
	Suppose that \cref{assumptions_F} holds.
	Then $\Phi_\tau$ and $R$ satisfy \cref{bound_third_diff} for all $\tau \in [0,\tau_0]$ with
	\begin{equation*}
		K_\tau := \tau \left(C_3 \left(\norm{A} M + \eunorm{y_\tau}\right) + 3 C_2 \bignorm{A\Sigma_\tau^\frac12}\right) + \frac{\tau^2}{2}\left(C_3 C_0 + 3 C_2 C_1\right).
	\end{equation*}
\end{proposition}
\begin{proof}
We express the scaled negative log-likelihood for all $x \in \R^d$ as
\begin{align*}
	\Phi_\tau(x) &= \frac12\eunorm{Ax + \tau F(x) - y_\tau}^2 \\
	&= \frac12\eunorm{Ax - y_\tau}^2 + \tau\scalprod{Ax - y_\tau, F(x)} + \frac{\tau^2}{2}\eunorm{F(x)}^2 \\
	&= \Phi_0(x) + \tau \Psi_{1}(x; \tau) + \tau^2\Psi_2(x),
\end{align*}
where $\Psi_{1}(x; \tau) := \scalprod{Ax - y_\tau, F(x)}$ and $\Psi_2(x) := \frac12\eunorm{F(x)}^2$.
For the first term, we have $D^3\Phi_0(x) = 0$ for all $x \in \Rd$ due to the linearity of $A$.
The third differentials of $\Psi_{1},\Psi_2$ can be stated explicitly as
\begin{align*}
	D^3\Psi_{1}(x; \tau)(h_1,h_2,h_3) &= \scalprod{D^3F(x)(h_1,h_2,h_3), Ax - y_\tau} + \scalprod{D^2F(x)(h_1,h_2), Ah_3} \\
	&\quad + \scalprod{D^2F(x)(h_2,h_3), Ah_1} + \scalprod{D^2F(x)(h_1,h_3), Ah_2},
\end{align*}
and
\begin{align*}
	D^3\Psi_2(x)(h_1,h_2,h_3) &= \frac12\scalprod{D^3F(x)(h_1,h_2,h_3), F(x)} + \frac12\scalprod{D^2F(x)(h_1,h_2), DF(x)h_3} \\
	&\quad + \frac12\scalprod{D^2F(x)(h_2,h_3), DF(x)h_1} + \frac12\scalprod{D^2F(x)(h_1,h_3), DF(x)h_2}
\end{align*}
for all $x, h_1, h_2, h_3 \in \Rd$ and $\tau\geq 0$.
Therefore, we have
\begin{equation*}
\begin{split}
	\norm{D^3\Psi_{1}(x; \tau)}_{\Sigma_\tau} &\le \norm{D^3F(x)}_{\Sigma_\tau} \eunorm{Ax - y_\tau} + 3\norm{D^2F(x)}_{\Sigma_\tau} \bignorm{A\Sigma_\tau^\frac12} \\
	&\leq C_3 \left(\norm{A} M + \eunorm{y_\tau}\right) + 3 C_2 \bignorm{A\Sigma_\tau^\frac12}
\end{split}
\end{equation*}
for all $x \in \Rd$ and $\tau \leq \tau_0$ by \cref{assumptions_F}.
Moreover, we obtain
\begin{align*}
	\norm{D^3\Psi_2(x)}_{\Sigma_\tau} &\le \frac12\norm{D^3F(x)}_{\Sigma_\tau} \eunorm{F(x)} + \frac32\norm{D^2F(x)}_{\Sigma_\tau}\norm{DF(x)}_{\Sigma_\tau} \\
	&\le \frac12 C_3 C_0 + \frac32 C_2 C_1
\end{align*}
for all $x \in \Rd$ and $\tau \leq \tau_0$.
Now, it follows that
\begin{align*}
	\norm{D^3\Phi_\tau(x)}_{\Sigma_\tau} &\le \tau \norm{D^3\Psi_1(x; \tau)}_{\Sigma_\tau} + \tau^2\norm{D^3\Psi_2(x)}_{\Sigma_\tau} \\
	&\le \tau \left(C_3 \left(\norm{A} M + \eunorm{y_\tau}\right) + 3 C_2 \bignorm{A\Sigma_\tau^\frac12}\right)
+ \frac{\tau^2}{2}\left(C_3 C_0 + 3 C_2 C_1\right)
\end{align*}
for all $x \in \Rd$ and $\tau \leq \tau_0$. 
\end{proof}

\subsection{Proof of Theorem \ref{conv_rate_nonlin_perturb}}

\begin{proof}[Proof of \cref{conv_rate_nonlin_perturb}]
First of all, we note that \cref{unique_minimizer} holds for all $\tau \le \tau_0$ by definition of $G_\tau$, $R$ and by \cref{ex_min_small_tau}.
Second of all, we note that \cref{quadratic_bound_I} holds for all $\tau \le \tau_0$ by \cref{Gt_delta} with $\delta_\tau$ as defined in \cref{Gt_delta}, and that $\delta_0 = \lim_{\tau \to 0} \delta_\tau > 0$.
This allows us to choose $\tau_1 \le \tau_0$ such that $\delta_\tau \ge \frac12\delta_0 =: \underline{\delta}$ for all $\tau \in [0,\tau_1]$.
Third of all, we note that \cref{bound_third_diff} holds for all $\tau \le \tau_0$ by \cref{Gt_K} with $\derbound_\tau$ as defined in \cref{Gt_K}, and that $\lim_{\tau \to 0} \derbound_\tau = 0$.

Since $\{\delta_\tau\}_{\tau \in [0,\tau_1]}$ is bounded from below, the left hand side of condition \eqref{cond_E2_dim} from \cref{nonasymp_est_dim} is bounded from below by $\kappa_1/K_\tau$ for some $\kappa_1 > 0$, and the right hand side of condition \eqref{cond_E2_dim} is bounded from above by $(8\ln \kappa_2/K_\tau)^{3/2}$ for small enough $\tau$ and some $\kappa_2 > 0$.
Therefore, we can choose $\tau_2 \le \tau_1$ such that condition \eqref{cond_E2_dim} is satisfied for all $\tau \in [0,\tau_2]$.
Now, \cref{nonasymp_est_dim} yields that
\begin{align*}
	\dTV\left(\mu^{y_\tau}, \Laplace_{\mu^{y_\tau}}\right) &\le \frac23\sqrt{2}e(1 + \epsilon)\epsilon^\frac12 \frac{\Gamma\big(\frac{d}{2} + \frac32\big)}{\Gamma\big(\frac{d}{2}\big)} \\
	&\quad\cdot\left[\left(C_3\left(\norm{A}M + \eunorm{y_t}\right) + 3C_2\bignorm{A\Sigma_\tau^\frac12}\right)\tau + \frac12\left(C_3 C_0 + 3 C_2 C_1\right)\tau^2\right]
\end{align*}
for all $\tau \in [0,\tau_2]$.	
By the convergence of $y_\tau$ and $\Sigma_\tau$, we can, moreover, choose $\tau_3 \le \tau_2$ such that both $\{\eunorm{y_\tau}\}_{\tau \in [0,\tau_3]}$ and $\{\norm{A\Sigma_\tau^{1/2}}\}_{\tau \in [0,\tau_3]}$ are bounded.
\end{proof}


\section{Outlook}

In this paper we prove novel error estimates for the Laplace approximation when applied to nonlinear Bayesian inverse problems. Here, the error is measured in TV distance and our estimates aim to quantify effects independent of the noise level.
Our central error estimate in \cref{main_estimate} is of particular use for high-dimensional problems because it can be evaluated without integrating in $\Rd$.
Our estimate in \cref{nonasymp_est_dim} makes the influence of the noise level, the nonlinearity of the forward operator, and the problem dimension explicit.
Our estimate for perturbed linear problems in \cref{conv_rate_nonlin_perturb}, in turn, specifies in more detail how the properties of the nonlinear perturbation affect the approximation error.

We point out that our central estimate diverges with an increasing dimension for a fixed noise level and forward mapping, and therefore such asymptotics does not provide any added value compared to the trivial TV upper bound of $1$. This unsatisfactory observation is natural since the limiting posterior and Laplace approximation (if well-defined) are singular with respect to each other and, consequently, the TV distance is maximized.
Future study is therefore needed to establish similar bounds with distances that metrize the weak convergence such as the $1$-Wasserstein distance. Such effort would be aligned with recent developments in BvM theory that extend to nonparametric Bayesian inference and, in particular, Bayesian inverse problems.


\bibliography{references}

\end{document}